\documentclass[a4paper,10pt]{article}
\usepackage[utf8]{inputenc}
\usepackage{amsthm,amssymb,amsmath}
\usepackage{hyperref}
\usepackage{xcolor}
\newtheorem{theorem}{Theorem}
\newtheorem{lemma}[theorem]{Lemma}

\newtheorem{conjecture}[theorem]{Conjecture}

\title{On Vizing's problem for triangle-free graphs}

\author{
Ross J. Kang\thanks{Korteweg--de Vries Institute for Mathematics, University of Amsterdam, Netherlands. Partially supported by Vidi (639.032.614) and Open Competition (OCENW.M20.009) grants of the Dutch Research Council (NWO) as well as  the Gravitation programme NETWORKS (024.002.003) of the Dutch Ministry of Education, Culture and Science (OCW). \protect\href{mailto:ross.kang@gmail.com}{\protect\nolinkurl{ross.kang@gmail.com}}}
\and
Matthieu Rosenfeld\thanks{LIRMM, Univ. Montpellier, CNRS, Montpellier, France. \protect\href{mailto:matthieu.rosenfeld@lirmm.fr}{\protect\nolinkurl{matthieu.rosenfeld@lirmm.fr}}.}
}

\begin{document}

\maketitle

\begin{abstract}
We prove that $\chi(G) \le \lceil (\Delta+1)/2\rceil+1$ for any triangle-free graph $G$ of maximum degree $\Delta$ provided $\Delta \ge 524$.
This gives tangible progress towards an old problem of Vizing, in a form cast by Reed.
We use a method of Hurley and Pirot, which in turn relies on a new counting argument of the second author.
\end{abstract}


\section{Introduction}\label{sec:intro}

Over half a century ago, Vizing~\cite{Viz68russian} wrote (this is a translation from Russian~\cite{Viz68translation} adapted to modern symbol usage, with $\omega$, $\Delta$, and $\chi$ denoting clique number, maximum degree, and chromatic number, respectively):
\begin{quote}\em
    `If $\Delta(G)$ is the maximum degree of a vertex in a graph $G$, it is clear that $\chi(G)\le \Delta(G) + 1$. Brooks showed in 1941 that $\chi(G)\le \Delta(G)$ whenever $\Delta(G) \ge 3$ and $\omega(G) \le \Delta(G)$. Further investigations could be conducted, taking into account a more exact relation between $\Delta$ and $\omega$. Perhaps one should start with estimates of the chromatic number of a graph without triangles ($\omega = 2$) and with given maximal degree for vertices.'
\end{quote}

This problem has its roots in even older questions and results about triangle-free graphs with arbitrarily high chromatic number~\cite{UnDe54,Zyk49,Myc55,Erd59}.
Deep connections with the off-diagonal Ramsey numbers, especially with the classic work of Ajtai, Koml\'os, Szemer\'edi~\cite{AKS80,AKS81} and Shearer~\cite{She83}, became evident through the seminal works of Johansson~\cite{Joh96+} and, more recently, Molloy~\cite{Mol19}.

\begin{theorem}[Molloy~\cite{Mol19}, cf.~\cite{Joh96+}]\label{thm:molloy}
As $\Delta\to\infty$, it holds that $\chi(G) \le (1+o(1))\Delta/\log\Delta$ for any triangle-free graph $G$ of maximum degree $\Delta$.
\end{theorem}

\noindent
This result has as a simple corollary an upper bound on the off-diagonal Ramsey number $R(3,k)$ that asymptotically (as $k\to\infty$) matches the longstanding bound of Shearer~\cite{She83}.

It transpires that the link between the task of finding either large independent sets or good colourings of triangle-free graphs and facets of the Lov\'asz local lemma~\cite{ErLo75}, such as the hard-core model (see~\cite{ScSo05}) or the entropy compression method~\cite{Mos09,MoTa10}, is fundamental. It underpins many recent developments in the area, see e.g.~\cite{Ber19,BBCK23,DJKP20,DKPS20+,DKPS24+,HuPi21+,HuPi23,Mol19}.

This work is no exception. To wit, a `counting trick' introduced recently by second author~\cite{Ros20} has shone a new light on the entropy compression method (see~\cite{WaWo22}). This `trick' is the main ingredient of an elegant new proof of Theorem~\ref{thm:molloy} devised by Hurley and Pirot~\cite{HuPi21+,HuPi23}. Here we observe how these recent developments yield tangible progress towards the old question above, in a sense that Vizing may have originally intended.
Our focus is the following conjecture\footnote{One could rightly argue that Vizing has asked for the value of \(
\chi_{\omega=2}(\Delta) := \sup\{\chi(G) : \text{$G$ is a triangle-free graph of maximum degree $\Delta$}\}\) for all $\Delta$, but here we take the more concrete Conjecture~\ref{conj:exact} as our point of reference, which clearly is sufficient challenging.}.

\begin{conjecture}\label{conj:exact}
It holds that $\chi(G) \le \lceil (\Delta+1)/2\rceil+1$ for any triangle-free graph $G$ of maximum degree $\Delta$.
\end{conjecture}

\noindent
While Theorem~\ref{thm:molloy} is only known to be sharp up to a multiplicative factor $2+o(1)$ (as $\Delta\to\infty$), the bound in Conjecture~\ref{conj:exact} is an `exact relation' in the sense that it is attained by the $4$-regular, $4$-chromatic Chv\'atal graph~\cite{Chv70} (as well as by, for example, odd cycles of length $5$ or more, the Petersen graph, and the Clebsch graph).
The conjecture is a special case ($\omega=2$) of Reed's conjecture~\cite{Ree98}, which asserts that any graph of clique number $\omega$ and maximum degree $\Delta$ has chromatic number at most $\lceil \frac12(\omega+\Delta+1) \rceil$. (The most significant general progress to date on Reed's conjecture is due to Hurley, de Joannis de Verclos and the first author~\cite{HJK22}.)
The bound in Conjecture~\ref{conj:exact} may have earlier provenance.
That is, Kostochka~\cite{Kos77}, motivated by Vizing's problem, proved the very same bound holds under the much stronger condition of girth being at least $4(\Delta+2)\log\Delta$. 
Borodin and Kostochka~\cite{BoKo77}, Catlin~\cite{Cat78a}, and Lawrence~\cite{Law78} independently, proved an upper bound of $\lceil 3(\Delta+1)/4\rceil$.
Kostochka (see~\cite[p.~83]{JeTo95} and also~\cite{Rab13}) proved an upper bound of $2\lceil (\Delta+2)/3 \rceil$.
Brooks' theorem~\cite{Bro41} implies the conjecture for $\Delta\le 4$.
We wish to separately emphasise the current smallest open case.
\begin{conjecture}[$\Delta=5$ of Conjecture~\ref{conj:exact}]\label{conj:exactsmallest}
Every $5$-regular triangle-free graph has chromatic number at most $4$.
\end{conjecture}
\noindent
This aligns with one of the few remaining cases of an old and largely refuted conjecture of Gr\"unbaum~\cite{Gru70}; he actually conjectured the negation of this.

Obviously Theorem~\ref{thm:molloy} implies the existence of some $\Delta_0$ such that the bound asserted in Conjecture~\ref{conj:exact} holds for all $\Delta\ge \Delta_0$. But to the best of our knowledge, no deliberate effort was made to explicitly bound $\Delta_0$. This is likely due to the technicality of earlier methods used to prove Theorem~\ref{thm:molloy} and its precursors. The situation has changed
and the result we highlight in this note is as follows.

\begin{theorem}\label{thm:main}
The bound asserted in Conjecture~\ref{conj:exact} holds for all $\Delta \ge 524$.
\end{theorem}

\noindent
We prove this by adapting the method of Hurley and Pirot to this version of the triangle-free colouring problem. We note that one straightforwardly derives an estimate of around $30000$ for $\Delta_0$ from Hurley and Pirot's main result~\cite[Thm.~2.10]{HuPi23}.
In showing Theorem~\ref{thm:main}, we wanted to indicate how a satisfying resolution to Vizing's triangle-free problem, as interpreted through Conjecture~\ref{conj:exact} and Reed's conjecture, is already within a stone's throw.
We would be interested to see better estimates on $\Delta_0$, although we do not expect anything below around $100$ without significant new ideas.

Later in the note we consider some related problems for bipartite or $C_4$-free graphs.

\subsection{Notation and preliminaries}

We write $\deg(v)$ for the degree of $v\in V(G)$.

We will establish Theorem~\ref{thm:main} in a more general context than chromatic number, not only in terms of list colouring, but also in an even more refined context, that of local colouring (as in~\cite{DJKP20,DKPS20+}, for instance). For this, we need some extra terminology.
Our notation is mostly standard; we collect it here for convenience.

Given a graph $G$, a list-assignment of $G$ is a mapping $L: V(G)\to 2^{{\mathbb N}}$.
Given a positive integer function on the vertices $k: V(G) \to\mathbb{N}$, we say that $L$ is a $k$-list-assignment if $|L(v)| \ge k(v)$ for all $v\in V(G)$.
Given a colouring $c: V(G) \to\mathbb{N}$ of $G$, we say that $c$ is an $L$-colouring if $c(v)\in L(v)$ for all $v\in V(G)$.
We will be preoccupied with proper colourings $c$, that is, with $c(v)\ne c(v')$ for any edge $vv'\in E(G)$, and sometimes we will omit the term proper if there is no possible confusion.

Recall that the list chromatic number $\chi_\ell(G)$ of $G$ is the least integer $k$ such that $G$ admits an $L$-colouring for any $k$-list-assignment $L$ of $G$.

We also work with partial colourings $c: S\to \mathbb{N}$ for some $S\subsetneq V(G)$.
Given a (partial) proper $L$-colouring $c$ of $G$ and $v\in V(G)$, we write $L_c(v)$ for the set of colours from $L(v)$ that are still available to $v$, that is, $L_c(v) = L(v) \setminus \{c(u) : u\in N(v)\}$.
For any subgraph $H\subseteq G$, we write $C_L(H)$ for the set of proper $L$-colourings of $H$.
In Section~\ref{sec:main}, it will be essential to have good control on how the order of $C_L$ evolves.

%
%
In the next section, we need the following simple technical result that will allow us to apply Jensen's inequality in the proof of our main result.

\begin{lemma}\label{convexity}
 For all $c>0$, $a>0$ and $b>1$, the function $z\mapsto \min\{0,cz(1-\frac{1}{b})^{ab/z}\}$ is convex. 
\end{lemma}
\begin{proof}
  This is easily verified by differentiating twice. For $z\in (0,+\infty)$, the second derivative is
  $$z\mapsto c\frac{\left(1-\frac{1}{b}\right)^{ab/z}\left(ab\ln\left(1-\frac{1}{b}\right)\right)^2}{z^3},$$
  which is positive and vanishes as $z\downarrow 0$.
\end{proof}

\section{Main bound}\label{sec:main}

With the above terminology, we can fully state and prove our main result.

\begin{theorem}\label{thm:mainlocal}
Define $k : x \mapsto \lceil (x+1)/2\rceil+1$.
Let $G$ be a triangle-free graph of minimum degree at least $524$.
Then $G$ admits a proper $L$-colouring for any $(k\circ \deg)$-list-assignment $L$ of $G$.
\end{theorem}
\noindent
Note that this statement directly implies Theorem~\ref{thm:main} even with the stronger parameter of list chromatic number $\chi_\ell(G)$ of $G$.
Similarly, Theorem~\ref{thm:mainlocal} but without the minimum degree condition would be a much stronger form of Conjecture~\ref{conj:exact}.
Theorem~\ref{thm:mainlocal} will follow immediately by combining the following two lemmata.

We remark that we have made a choice of $k$ here to align with Conjecture~\ref{conj:exact}, but we could very well have chosen $x \mapsto \lceil 3(x+1)/4\rceil$ or $x \mapsto 2\lceil (x+2)/3 \rceil$, say, and obtained ---under a similar minimum degree condition, and through suitable though unilluminating technical adjustments--- the same conclusion.

For any positive integers $\Delta_0$, $\ell$, $t$ and a non-decreasing function $k: \mathbb{N}\rightarrow\mathbb{N}$, we say that $(\Delta_0,\ell,t, k)$ has Property~$\mathbf{(P)}$ if for all integers $\delta\ge\Delta_0$,
\begin{enumerate}
  \item $0<t< \ell< k(\delta)<\delta$, and
  \item $\left(k(\delta)-\frac{t\delta}{\ell}\right)\left(1-\frac{1}{t+1}\right)^{\frac{(t+1)(\delta-t\delta/\ell)}{ k(\delta)-t\delta/\ell}}\ge \ell\,.$
\end{enumerate}

\begin{lemma}
Let  $\Delta_0$, $\ell$ and $t$ be positive integers and $k: \mathbb{N}\to\mathbb{N}$ be a non-decreasing function such that $(\Delta_0,\ell, k,t)$ has Property~$\mathbf{(P)}$.
Let  $G$ be a triangle-free graph with minimum degree
at least $\Delta_0$, and let $L$ be a $(k\circ \deg)$-list-assignment of $G$. Then, for any $v\in V(G)$,
\begin{align*}|C_L(G)|/|C_L(G-v)|\ge \ell\,.\end{align*}
\end{lemma}
For brevity, we allow ourselves to write $k$ instead of $k\circ \deg$.

\begin{proof}
The proof is by induction on the order of the vertex set $V(G)$. The case $V(G)=\{v\}$ is trivial. We suppose that our claim holds for all strict subgraphs of $G$ and our goal is to prove that this holds for $G$.
We consider what happens if we choose a colouring $\mathbf{c}$ uniformly at random from $C_L(G-v)$.

Our induction hypothesis implies that for all $u\in V(G)-v$,
\begin{equation*}
|C_L(G-v)|\ge \ell |C_L(G-v-u)|\,.
\end{equation*}

Fix a vertex $u\in N(v)$. Let $F=\{c\in C_L(G-v): |L_c(u)|\le t\}$. The restriction of any coloring $c\in F$ to $G-v-u$ is in $C_L(G-v-u)$. Moreover, any coloring from $C_L(G-v-u)$ is the restriction of at most $t$ such colourings, hence, $|F|\le t|C_L(G-v-u)|$. We deduce that
$$\mathbb{P}(|L_\mathbf{c}(u)|\le t)= \frac{|F|}{|C_L(G-v)|}\le\frac{t|C_L(G-v-u)|}{\ell|C_L(G-v-u)|}= \frac{t}{\ell}\,.$$

Let $\mathbf{t}_v=\left|\{u\in N(v): |L_\mathbf{c}(u)|\le t\}\right|$. Then by linearity of expectation 
\begin{equation}\label{expectedsizeofB}
\mathbb{E}[\mathbf{t}_v]\le \frac{t\deg(v)}{\ell}\,.
\end{equation}

Let $G_0= G-v-N(v)$, and consider the distribution of $\mathbf{c}$ conditioned on $\mathbf{c}|_{G_0}=c_0$ for some $c_0 \in C_L(G_0)$. Note that for all $u\in N(v)$, $L_\mathbf{c}(u)= L_{c_0}(u)$. Conditioned on $\mathbf{c}|_{G_0}=c_0$, the colours of the neighbours of $v$ are conditionally independent and uniformly chosen from the respective sets $L_{c_0}(u)$.

The problem is reduced to the following elementary statement. We are given sets $L_{\mathbf{c}}(u)$ of colours, for each $u\in N(v)$. We are also given random variables $\mathbf{X}_u\in L_{\mathbf{c}}(u)$, for each $u\in N(v)$, each chosen independently and uniformly from the respective set.
Writing $\mathbf{X}= |L(v)-\{\mathbf{X}_u: u\in N(v)\}|$ (note that it is possible that $\mathbf{X}_u\not\in L(v)$), we have that
 $\frac{|C_L(G)|}{|C_L(G-v)|} =\mathbb{E}[|L_\mathbf{c}(v)|]= \mathbb{E}[\mathbf{X}] $. The problem is thus reduced to showing that  $\mathbb{E}[\mathbf{X}]\ge \ell$.

In order to bound $\mathbb{E}[\mathbf{X}]$, we condition on the values of $\mathbf{X}_u$ for each $u$ such that $|L_{\mathbf{c}}(u)|\le t$, and we write $B$ for the (random) set of those values. 
For each $j\in L(v)$, we write the indicator variable $\mathbf{Y}_j= [j\in \mathbf{X}]$, hence $\mathbf{X}\ge\sum\limits_{j\in L(v)\setminus B}\mathbb{E}[\mathbf{Y}_j]$.
 By linearity of expectation, the conditional expectation of $\mathbf{X}$ given $c_0$ satisfies
\begin{align*}
\mathbb{E}[\mathbf{X} \mid \mathbf{c}|_{G_0}=c_0]
&\ge\sum_{j\in L(v)\setminus B}\prod_{\substack{u \in N(v)\\L_{c_0}(u)\ni j\\ |L_{c_0}(u)|>t}}\left(1-\frac{1}{|L_{c_0}(u)|}\right)\\
&\ge(k(v)-|B|)\left(\prod_{j\in L(v)\setminus B}\prod_{\substack{u \in N(v)\\L_{c_0}(u)\ni j\\ |L_{c_0}(u)|>t}}\left(1-\frac{1}{|L_{c_0}(u)|}\right)\right)^{1/(k(v)-|B|)}\\
&\ge(k(v)-|B|)\left(\prod_{\substack{u\in N(v)\\ |L_{c_0}(u)|>t}}\prod_{j\in L_{c_0}(u)\cap L(v)\setminus B}\left(1-\frac{1}{| L_{c_0}(u)|}\right)\right)^{1/(k(v)-|B|)}\\
&\ge(k(v)-|B|)\left(\prod_{\substack{u\in N(v)\\ |L_{c_0}(u)|>t}}\left(1-\frac{1}{| L_{c_0}(u)|}\right)^{| L_{c_0}(u)|}\right)^{1/(k(v)-|B|)}\\
&\ge(k(v)-|B|)\left(1-\frac{1}{t+1}\right)^{\frac{(t+1)(\deg(v)-|B|)}{k(v)-|B|}}\\
&=(k(v)-|B|)\left(1-\frac{1}{t+1}\right)^{(t+1)}\left(1-\frac{1}{t+1}\right)^{\frac{(t+1)(\deg(v)-k(v))}{k(v)-|B|}}\,,
\end{align*}
where the second inequality is an application of the AM--GM inequality.

By Lemma~\ref{convexity} (with $c=(1-1/(t+1))^{t+1}$, $a= \deg(v)-k(v)$ and $b=t+1$), the minimum of zero and this last quantity is a convex function of $k(v)-|B|$. 
Thus, we can use Jensen's inequality to average over all values of $c_0$ as follows:
\begin{align*}
  \mathbb{E}[\mathbf{X}]&\ge\mathbb{E}\left[\min\left\{0,(k(v)-|B|)\left(1-\frac{1}{t+1}\right)^{\frac{(t+1)(\deg(v)-|B|)}{k(v)-|B|}}\right\}\right]\\
  &\ge\min\left\{0,\mathbb{E}\left[k(v)-|B|\right]\left(1-\frac{1}{t+1}\right)^{\frac{(t+1)(\deg(v)-k(v)+\mathbb{E}\left[k(v)-|B|\right])}{\mathbb{E}\left[k(v)-|B|\right]}}\right\}.
\end{align*}

By definition $|B|\le\mathbf{t}_v$. Hence, by equation~\eqref{expectedsizeofB}, we have
$$\mathbb{E}\left[k(v)-|B|\right]\ge k(v)-\frac{t\deg(v)}{\ell}\,.$$ 

This and the previous series of inequalities imply
 \[\mathbb{E}[\mathbf{X}]\ge\left(k(v)-\frac{t\deg(v)}{\ell}\right)\left(1-\frac{1}{t+1}\right)^{\frac{(t+1)(\deg(v)-t\deg(v)/\ell)}{ k(v)-t\deg(v)/\ell}}\,,\]
which by our theorem hypothesis implies $\mathbb{E}(\mathbf{X})\ge \ell$ as desired.   
\end{proof}

\begin{lemma}
Let $\Delta_0=524$, $\ell=8$, $t=1$ and $k$ be the function such that for all $\delta\in\mathbb{N}$, $k=\lceil\frac{\delta+1}{2}\rceil+1$. The quadruplet $(\Delta_0,\ell,t, k)$ has Property~$\mathbf{(P)}$.
\end{lemma}
\begin{proof}
By definition of $k$, $t$ and $\ell$, we have for all $\delta>0$ that
\begin{itemize}
  \item $k(\delta)-t\delta/\ell\ge \frac{3\delta}{8}+1$, and
  \item $\frac{\delta-t\delta/\ell}{k(\delta)-t\delta/\ell}\le\frac{1-t/\ell}{3/8}=\frac{7}{3}$.
\end{itemize}
We can now verify that 
\begin{align*}
  \left(k(\delta)-\frac{t\delta}{\ell}\right)\left(1-\frac{1}{t+1}\right)^{\frac{(t+1)(\delta-t\delta/\ell)}{ k(\delta)-t\delta/\ell}}
  \ge\left(\frac{3\delta}{8}+1\right)2^{-14/3}
  \ge8.01>\ell\,,
\end{align*}
where we use the fact that $\delta\ge540$ for the second inequality.

For each of the remaining cases, with $524\le\delta<540$, one can (for instance with the help of a computational mathematical engine) verify the inequality by direct computation. This completes the establishment of Property~$\mathbf{(P)}$.
\end{proof}

\section{The bipartite case}\label{sec:bipartite}

Conjecture~\ref{conj:exact} is quite trivial in the special case of bipartite $G$. 
But in consideration of Theorem~\ref{thm:mainlocal}, it is perhaps insightful to point out how the following result holds as a consequence of polynomial methods.
We thank Zden\v{e}k Dvo\v{r}\'{a}k and Ronen Wdowinski for pointing us to this implication.

\begin{theorem}[\cite{AlTa92}]\label{thm:zdenekobs}
For $k : x \mapsto \lceil x/2\rceil+1$,
any bipartite graph $G$ admits a proper $L$-colouring for any $(k\circ \deg)$-list-assignment $L$ of $G$.
\end{theorem}

\begin{proof}
There is an orientation of $G$ in which the outdegree is at most $ \lceil \deg(v)/2\rceil$ for each $v\in V(G)$. For example, add a vertex $v'$ adjacent to all vertices of odd degree, take any Eulerian orientation of the resulting supergraph, and then remove $v'$. Moreover, the orientation contains no odd directed cycle, since $G$ is bipartite. The result then follows by a result of Alon and Tarsi~\cite[Thm.~1.1]{AlTa92}.
\end{proof}
\noindent
Since $K_{3,3}$ has list chromatic number $3$, this result is best possible.
Moreover, the conclusion of Theorem~\ref{thm:zdenekobs} fails if we take $k$ to be {\em any} sublinear function, due to a simply defined tree construction with a very large span of degrees~\cite[Prop.~11]{DJKP20}.
Theorem~\ref{thm:zdenekobs} is related to a vexing conjecture of Alon and Krivelevich~\cite{AlKr98}, which posits a bound on the list chromatic number logarithmic in $\Delta$.
We have implicitly conjectured (shortly after stating Theorem~\ref{thm:mainlocal}) an essentially stronger form of Theorem~\ref{thm:zdenekobs} for $G$ being triangle-free.

Next, as a means of comparison, we focus on an assertion that is in fact weaker than Theorem~\ref{thm:zdenekobs}. For this, we introduce some extra terminology. Given a bipartite graph $G=(V{=}A\cup B,E)$ with parts $A$, $B$ and positive integers $k_A$, $k_B$, a mapping $L: A\to \binom{{\mathbb Z}^+}{k_A},B\to \binom{{\mathbb Z}^+}{k_B}$ is called a {\em $(k_A,k_B)$-list-assignment} of $G$. We say $G$ is {\em $(k_A,k_B)$-choosable} if there is guaranteed a proper $L$-colouring of $G$ for any such $L$. It is natural to permit maximum degree constraints $\Delta_A$ and $\Delta_B$ that vary per part, and then ask for bounds on $k_A$ and $k_B$ (in terms of $\Delta_A$ and $\Delta_B$) that guarantee $(k_A,k_B)$-choosability. This problem was introduced in~\cite{ACK21}, but in the current context we highlight the following corollary of Theorem~\ref{thm:zdenekobs}.

\begin{theorem}\label{mainBipConj}
  Let $\Delta_A,\Delta_B$ be positive integers.
  Then any bipartite graph $G=(V{=}A \cup B, E)$ with parts $A$ and $B$ having maximum degrees at most $\Delta_A$ and $\Delta_B$, respectively, is $(\lceil\Delta_A/2\rceil+1,\lceil\Delta_B/2\rceil+1)$-choosable.
\end{theorem}

\noindent
We note that modest adaptation of the method used in Section~\ref{sec:main} establishes Theorem~\ref{mainBipConj} for all but at most $274000$ pairs $(\Delta_A,\Delta_B)$. Our aim in this section is to illustrate for comparison how another method, given in~\cite{ACK21}, establishes Theorem~\ref{mainBipConj} for all but at most $27000$ pairs $(\Delta_A,\Delta_B)$. 
We straightforwardly apply the following result.

\begin{lemma}[{\cite[Thm.~4]{ACK21}}]\label{mainBip}
Let the positive integers $\Delta_A$, $\Delta_B$, $k_A$, $k_B$, with $k_A\le \Delta_A$ and $k_B\le \Delta_B$, satisfy one of the following conditions:
\begin{align}
\label{itm:transversals}
&k_B \ge (ek_A\Delta_B)^{1/k_A}\Delta_A; \ \text{ or }\\
\label{itm:coupon}
&e(\Delta_A(\Delta_B-1)+1) \left(1-(1-1/k_B)^{\Delta_A\min\left\{1,k_B/k_A\right\}}\right)^{k_A} \le 1.
\end{align}
Then any bipartite graph $G=(V{=}A\cup B,E)$ with parts $A$ and $B$ having maximum degrees at most $\Delta_A$  and $\Delta_B$, respectively, is $(k_A,k_B)$-choosable.
\end{lemma}

This result relies on two direct applications of the Lov\'asz Local Lemma.
We remark that of these two conditions, Condition~\ref{itm:coupon}, based on a ``coupon collector'' intuition, is the one of greater importance for our purposes here.

\begin{theorem}\label{thm:asym}
The conclusion of Theorem~\ref{mainBipConj} holds true under additionally either of the following conditions (or, symmetrically, the analogous conditions with $\Delta_A$ and $\Delta_B$ exchanged):
\begin{enumerate}
\item\label{itm:one}
$\Delta_A\ge165$ and $\Delta_A\ge\Delta_B\ge56$; or
\item\label{itm:other}
$\Delta_A\le55$ and $\Delta_B\ge153$.
\end{enumerate}
\end{theorem}

\begin{proof}
Throughout the proof, we let $k_A=\lceil\Delta_A/2\rceil+1$ and $k_B=\lceil\Delta_B/2\rceil+1$.

We begin with proving the assertion under Condition~\ref{itm:one}.
For this it suffices to check that~\eqref{itm:coupon} holds. 
We need the following inequalities:
\begin{align*}
\Delta_A\min\{1,k_B/k_A\} &=\Delta_A\frac{\lceil\Delta_B/2\rceil+1}{\lceil\Delta_A/2\rceil+1}\le\Delta_A\frac{\Delta_B/2+1.5}{\Delta_A/2}=\Delta_B+3 \text{ and }\\
k_B & \ge (\Delta_B+2)/2\,.
\end{align*}
Using these, we can give a first bound on the left-hand side of~\eqref{itm:coupon},
\begin{align*}
 \text{LHS of~\eqref{itm:coupon}} &\le e(\Delta_A(\Delta_B-1)+1)\left(1-\left(1-\frac{2}{\Delta_B+2}\right)^{\Delta_B+3}\right)^{k_A}\\
 &\le e\Delta_A^2 \left(1-\left(1-\frac{2}{\Delta_B+2}\right)^{\Delta_B+3}\right)^{k_A}\,.
\end{align*}
Computing the derivative of $\Delta_B\mapsto 1-\left(1-\frac{2}{\Delta_B+2}\right)^{\Delta_B+3}$, one finds that this function is decreasing for $\Delta_B \ge 0$ (and this easy exercise is left to the reader). Hence, using that $\Delta_B\ge56$, 
$$1-\left(1-\frac{2}{\Delta_B+2}\right)^{\Delta_B+3}\le1-\left(1-\frac{2}{58}\right)^{59}<0.874\,.$$
We then deduce the following simpler bound on the LHS of~\eqref{itm:coupon}:
\begin{align*}
  \text{LHS of~\eqref{itm:coupon}}
  &< e\Delta_A^2\cdot 0.874^{\lceil\Delta_A/2\rceil+1}
  \le e\Delta_A^2\cdot 0.874^{\Delta_A/2+1}\\
  &\le e\cdot 165^2\cdot 0.874^{83.5}<0.97<1,
 \end{align*}
 where the third inequality uses that by assumption $\Delta_A\ge165$ and that the previous expression is a decreasing function of $\Delta_A$ for $\Delta_A\ge35$ (which is verified by computing the derivative).
 So we can apply Lemma~\ref{mainBip} to reach the conclusion in this case.

Next we show how to derive the conclusion under Condition~\ref{itm:other}.
  The case $\Delta_A=1$ is trivial, so we assume $\Delta_A\ge2$.
In this case, we aim to check that~\eqref{itm:transversals} holds. It suffices to show that the following inequality holds:
  \begin{equation*}
  \Delta_B/2\ge (ek_A\Delta_B)^{1/k_A}\Delta_A\,,
  \end{equation*}
  which is equivalent to
  \begin{align}\label{interGoal}
    \Delta_B\ge 
    \left(ek_A(2\Delta_A)^{k_A}\right)^{1/(k_A-1)}\,.
  \end{align}
  It remains simply to compute this last expression for all $\Delta_A\ge 2$ under Condition~\ref{itm:other}; the results of this computation are provided in Table~\ref{table1}.
  As desired, the right-hand side is always at most $153\le\Delta_B$, establishing~\eqref{itm:transversals}. 
 So again we apply Lemma~\ref{mainBip} to reach the conclusion in this case.
  \end{proof}


  \begin{table}
    \centering
    \begin{tabular}{c|lllllllll} 
      \hline
      $\Delta_A$ 
      &2
      &3
      &4
      &5
      &6
      &7
      &8
      &9
      &10
      \\\hline
      & 87
      & 42
      & 65
      & 48
      & 61
      & 52
      & 62
      & 57
      & 64
      \\\hline\hline
      $\Delta_A$ 
      &11
      &12
      &13
      &14
      &15
      &16
      &17
      &18
      &19
      \\\hline
      & 61
      & 67
      & 65
      & 70
      & 69
      & 74
      & 73
      & 78
      & 77
       \\\hline\hline
      $\Delta_A$ 
      &20
      &21
      &22
      &23
      &24
      &25
      &26
      &27
      &28
      \\\hline
      & 82
      & 81
      & 86
      & 86
      & 90
      & 90
      & 94
      & 94
      & 98
       \\\hline\hline
      $\Delta_A$      
      & 29
      & 30
      & 31
      & 32
      & 33
      & 34
      & 35
      & 36
      & 37
      \\\hline
      & 98
      & 102
      & 102
      & 106
      & 107
      & 110
      & 111
      & 114
      & 115
       \\\hline\hline
      $\Delta_A$      
      & 38
      & 39
      & 40
      & 41
      & 42
      & 43
      & 44
      & 45
      & 46
      \\\hline
      & 118
      & 119
      & 122
      & 123
      & 127
      & 128
      & 131
      & 132
      & 135
       \\\hline\hline
      $\Delta_A$      
      & 47
      & 48
      & 49
      & 50
      & 51
      & 52
      & 53
      & 54
      & 55
      \\\hline
      & 136
      & 139
      & 140
      & 143
      & 144
      & 147
      & 148
      & 151
      & 153
      \\\hline
    \end{tabular}
    \caption{Evaluations of $\left\lceil\left(ek_A(2\Delta_A)^{k_A}\right)^{1/(k_A-1)}\right\rceil$ for $\Delta_A\in\{2,\dots,55\}$.
    \label{table1}}
  \end{table}


In this section, we made Theorem~\ref{mainBipConj} a baseline for comparison due to our objective in Conjecture~\ref{conj:exact}. From this viewpoint, the method via Theorem~\ref{thm:zdenekobs} appears most effective. 
However, both the method used for Theorem~\ref{thm:asym} and the one in Section~\ref{sec:main} can be used to establish $(k_A,k_B)$-choosability for $k_A=(1+o(1))\Delta_A/\log \Delta_A$ and $k_B=(1+o(1))\Delta_B/\log \Delta_B$~\cite{ACK21,HuPi21+,HuPi23}, which seems out of reach for the Alon--Tarsi method. Moreover, the method used for Theorem~\ref{thm:asym} depends critically on the bipartition, and, while the method does not obviously transfer to triangle-free graphs, it handles more asymmetric dependencies of $k_A$ and $k_B$ upon $\Delta_A$ and $\Delta_B$, as in~\cite{ACK21}.
Thus each  method has its strength.

\section{Concluding remarks}\label{sec:conclusion}

We have demonstrated how Conjecture~\ref{conj:exact} is temptingly within reach. We encourage further research on this.
We end with yet another related problem.

\begin{conjecture}\label{conj:C4}
It holds that $\chi(G) \le \lceil (\Delta+1)/2\rceil+1$ for any $C_4$-free graph $G$ of maximum degree $\Delta$.
\end{conjecture}

\noindent
This bound is attained for the line graph of any snark of girth at least $5$.
Catlin~\cite{Cat78b} proved an upper bound of $2(\Delta+3)/3$.
An $O(\Delta/\log\Delta)$ bound was observed in~\cite[Cor.~2.4]{AKS99}, which was recently tightened to a $(1+o(1)) \Delta/\log\Delta$ bound in~\cite[Thm.~4]{DKPS20+} (in fact for any forbidden cycle length $\ell = \Delta^{o(1)}$ as $\Delta\to\infty$).
One can derive from~\cite[Thm.~2.10]{HuPi23} that the bound in Conjecture~\ref{conj:C4} holds for all $\Delta$ at least around $60000$.
We remark that while the fractional chromatic number version of  Conjecture~\ref{conj:exact} is known (see~\cite[Sub.~21.3]{MoRe02}), the fractional analogue of Conjecture~\ref{conj:C4} may already be nontrivial.


\paragraph{Acknowledgements.}
In previous versions of this note, we had posed Theorem~\ref{thm:zdenekobs} as a conjecture, and we thank Zden\v{e}k Dvo\v{r}\'{a}k and Ronen Wdowinski for indicating to us its earlier solution.
We are grateful to the anonymous referee for their careful review and helpful comments.
We are grateful to Morteza Hasanvand for pointing out the sharpness in Conjecture~\ref{conj:C4}.

\paragraph{Open access statement.} For the purpose of open access,
a CC BY public copyright license is applied
to any Author Accepted Manuscript (AAM)
arising from this submission.

\bibliographystyle{abbrv}
\bibliography{vizingq}

\end{document}